\documentclass{amcjoucc}

\usepackage{hyperref}
\usepackage{amssymb}
\usepackage[capitalize]{cleveref}
\usepackage{enumitem}
    \setlist{itemsep=\smallskipamount}
	\setlist[enumerate, 1]{label =\textup{(\arabic*)}, ref=\arabic*}
	\setlist[enumerate, 2]{label =\textup{(\alph*)}, ref=\theenumi\alph*, topsep=\smallskipamount}
	\setlist[itemize, 2]{label=$\circ$, topsep=\smallskipamount}
\usepackage{mathtools} 

\DeclareMathOperator{\Cay}{Cay}

\newcommand{\Z}{\mathbb{Z}}

\newcommand{\0}{{0}}
\newcommand{\1}{\mathbf{1}}
\renewcommand{\and}{\mathbin{\&}}

\newcommand{\compose}{\mathbin{\circ}}
\newcommand{\iso}{\cong}

\newcommand{\sgood}[1]{S^{#1}_{\!good}}
\newcommand{\var}[2]{X{}_{#1}^{#2}}

\newcommand{\pref}[1]{(\ref{#1})}
\newcommand{\fullref}[2]{\ref{#1}\pref{#1-#2}}
\newcommand{\fullcref}[2]{\cref{#1}\pref{#1-#2}}
\newcommand{\fullCref}[2]{\Cref{#1}\pref{#1-#2}}
\newcommand{\csee}[1]{(see \cref{#1})}
\newcommand{\fullcsee}[2]{(see \fullcref{#1}{#2})}
\newcommand{\arxiv}[1]{\href{https://arxiv.org/abs/#1}{arXiv:#1}}

\newcommand{\MR}[1]{\href{https://mathscinet.ams.org/mathscinet-getitem?mr=#1}{MR\,#1}}
\newcommand{\doi}[1]{\href{https://doi.org/#1}{doi:#1}}

\numberwithin{equation}{section}
\newtheorem{cor}[equation]{Corollary}
\newtheorem{lem}[equation]{Lemma}
\newtheorem{prop}[equation]{Proposition}
\newtheorem{thm}[equation]{Theorem}

\Crefname{lem}{Lemma}{Lemmas}
\Crefname{thm}{Theorem}{Theorems}

\theoremstyle{definition}
\newtheorem*{ack}{Acknowledgments}
\newtheorem{defn}[equation]{Definition}

\newtheorem{rem}[equation]{Remark}
\newtheorem{rems}[equation]{Remarks}

\newcounter{case}

\renewcommand{\thecase}{\arabic{case}}
\crefformat{case}{Case~#2#1#3}
\Crefname{case}{Case}{Cases}
\crefname{case}{case}{cases}

\newcounter{stepholder} 
\newcounter{step}
\newenvironment{step}{\refstepcounter{step}\bf\sffamily
\par\medskip \indent Step \thestep.\ \it\ignorespaces}{\unskip\upshape}
\renewcommand{\thestep}{\arabic{step}}
\crefformat{step}{Step~#2#1#3}
\Crefname{step}{Step}{Steps}
\crefname{step}{case}{cases}

\makeatletter
\newcommand{\noprelistbreak}{\smallskip\@nobreaktrue\nopagebreak} 
\makeatother

\makeatletter
\@finaltrue
\renewenvironment{frontmatter}
{\thispagestyle{plain}}
{\vskip 20pt%
\blfootnote{\raggedright \ifnum\@authorcount=1\textit{E-mail address:}\else\textit{E-mail addresses:}\fi ~\@emails}%
\if@proofsline\global\linenumbers\fi%
}

\renewenvironment{abstract}
{\hrule height 0.25pt
\vskip 5pt
\noindent \textbf{Abstract}
\vskip 5pt
}
{
\vskip 5pt
\noindent \textit{\small Keywords:~\@keywords}

\vskip 3pt
\noindent \textit{\small Math.\ Subj.\ Class. (2020): \@msc}
\vskip 5pt
\hrule height 0.25pt
}
\def\@oddrunninghead{Isomorphisms of edge-coloured abelian Cayley graphs}
\def\@evenrunninghead{\@oddrunninghead}
\makeatother

\begin{document}

\begin{frontmatter} 

\titledata{Isomorphisms of abelian Cayley graphs 
	\\ with their natural edge-colouring}{}

\authordatatwo{Shirin Alimirzaei}
{shirin.alimirzaei@uleth.ca}  
{0000-0002-9983-7462}    
{} 
{Dave Witte Morris}%
{dmorris@deductivepress.ca,
    https://deductivepress.ca/dmorris}%
{} 
{} 
{Department of Mathematics and Computer Science, 
University of Lethbridge, 
\\ 
4401 University Drive, 
Lethbridge, Alberta, T1K~3M4, Canada}



\begin{abstract}
We prove that if $\varphi$ is an isomorphism between two connected Cayley graphs of abelian groups, and $\varphi$ respects the natural edge-colourings of the Cayley graphs, then $\varphi$ is the composition of a group isomorphism and a colour-preserving graph automorphism. This implies that if every colour-preserving automorphism of a connected abelian Cayley graph $\Cay(G;S)$ is an affine map, then the same is true for every colour-permuting automorphism. We also show that this property holds if and only if the subgroup generated by
    $ \{\, s \in S \mid 2s \neq c \,\} \cup \{c\}$
has index $\le 2$ for every element~$c$ of order~$2$ in~$G$.
\end{abstract}

\end{frontmatter}

\section{Introduction}

\begin{defn}[cf.\ {\cite[Defns.~1.1 and~1.2]{HKMM}}]
Let $S$ be a subset of an abelian group~$G$. If $S$ is \emph{symmetric} (which means $S = -S$), then the \emph{abelian Cayley graph} $\Cay(G; S)$ is the (undirected) graph whose vertices are the elements of~$G$, with an edge joining $g$ and~$g + s$ for all $g \in G$ and $s \in S$. (Cayley graphs can be defined from any group, but we add the adjective ``abelian'' to specify that the group is required to be abelian in this paper.)
This graph has a natural edge-colouring: the edge joining $g$ and~$g + s$ is coloured with the set $\{\pm s\}$.
	\begin{itemize}
	\item An automorphism $\varphi$ of $\Cay(G; S)$ is \emph{colour-preserving} if it preserves the colours of the edges. (In other words, for all $g \in G$ and $s \in S$, we have $\varphi(g + s) \in \{\varphi(g) \pm s\}$.) 
	\item An isomorphism $\varphi$ from one abelian Cayley graph $\Cay(G_1; S_1)$ to another abelian Cayley graph $\Cay(G_2; S_2)$ is \emph{chromatic} if it respects the edge-colourings. More concretely, there is a bijection $\pi \colon S_1 \to S_2$, such that $\varphi(g + s) \in \{\varphi(g) \pm \pi(s) \}$ for all $g \in G_1$ and $s \in S_1$.
	\end{itemize}
\end{defn}

It is easy to see that if $\tau \colon G_1 \to G_2$ is a group isomorphism, then it is a chromatic isomorphism from $\Cay(G_1; S_1)$ to $\Cay \bigl( G_2; \tau(S_1) \bigr)$, for any symmetric subset~$S_1$ of~$G_1$ (cf.\ \cite[p.~190]{HKMM}). It is also obvious that every colour-preserving automorphism is a chromatic isomorphism. The following result shows in the abelian case that every chromatic isomorphism is a combination of these two obvious types.

\begin{thm} \label{ChromIso}
If\/ $\Cay(G_1; S_1)$ and\/ $\Cay(G_2; S_2)$ are connected abelian Cayley graphs, then every chromatic isomorphism from $\Cay(G_1; S_1)$ to $\Cay(G_2; S_2)$ is the composition of a colour-preserving automorphism of\/ $\Cay(G_1; S_1)$ and a group isomorphism from $G_1$ to~$G_2$.
\end{thm}

This implies that every abelian group is determined (within the category of abelian groups) by any of its edge-coloured connected Cayley graphs:

\begin{cor} \label{ChromIsoImpliesGIso}
Assume $G_1$ and $G_2$ are abelian groups. If there is a chromatic isomorphism from some connected Cayley graph of~$G_1$ to some Cayley graph of~$G_2$, then $G_1$ is isomorphic to~$G_2$.
\end{cor}

\begin{rem} \label{NotSameGroup}
It is well known that \cref{ChromIsoImpliesGIso}'s assumption that $G_2$ is abelian cannot be removed. 
For example, let $D_{2n} = \langle a, b \mid a^n = b^2 = abab = \1 \rangle$ be the dihedral group of order~$2n$. Then $\Z_n \oplus \Z_2 \not\iso D_{2n}$ (if $n \ge 3$), but 
	\[ \text{$\Cay \bigl( \Z_n \oplus \Z_2; \{ \pm(1,0), (0,1)\} \bigr)$ is chromatically isomorphic to $\Cay \bigl( D_{2n}; \{a^{\pm1}, b\} \bigr)$.} \]

\end{rem}

\begin{defn}[cf.\ {\cite[Defns.~1.2 and 1.4]{HKMM}}]
\leavevmode
\noprelistbreak
\begin{itemize}
	\item A bijection $\varphi$ from an abelian group~$G_1$ to an abelian group~$G_2$ is \emph{affine} if it is the composition of a group isomorphism and a translation. More concretely, there is a group isomorphism $\tau \colon G_1 \to G_2$ and an element $g_2 \in G_2$, such that $\varphi(x) = \tau(x) + g_2$. Then there is also an element $g_1 \in G_1$, such that $\varphi(x) = \tau(x + g_1)$.
	(In this paper, all affine maps are bijections, but most authors do not have this restriction: they allow $\tau$ to be a homomorphism, rather than requiring it to be an isomorphism.)
	
	\item If an automorphism of an abelian Cayley graph is chromatic, then it is called a \emph{colour-permuting automorphism}.

	\item An abelian Cayley graph is:
		\begin{itemize}
		\item \emph{CCA} if every colour-preserving automorphism is affine;
		\item \emph{strongly CCA} if every colour-permuting automorphism is affine. 
		\end{itemize}
	\end{itemize}
\end{defn}

It is known that if every connected Cayley graph of a finite abelian group~$G$ is CCA, then every connected Cayley graph of~$G$ is strongly CCA \cite[Prop.~4.1]{HKMM}. (The converse is obvious.) The following \lcnamecref{Tcharacterization} strengthens this by establishing that the implication holds for each individual Cayley graph, not only when all of them are CCA (and also by removing the assumption that the group is finite). It also provides a concrete characterization of the abelian Cayley graphs that are CCA.

\begin{thm}\label{Tcharacterization}
If\/ $\Cay(G;S)$ is a connected abelian Cayley graph, then the following are equivalent:
\begin{enumerate}
   
    \item \label{Tcharacterization-CCA}
    $\Cay(G; S)$ is CCA.
    
    \item \label{Tcharacterization-strong}
    $\Cay(G; S)$ is strongly CCA.
     
    \item \label{Tcharacterization-index}
    $|G: \langle T(S,2s)\rangle| \le 2$ for all $s\in S$ with $|s|=4$, where 
        \[ T(S,c) \coloneqq \{\, t \in S \mid 2t \neq c \,\} \cup \{c\} . \]
 \end{enumerate}
\end{thm}

\begin{rems}
\leavevmode\noprelistbreak
    \begin{enumerate}
    \item The implication ($\ref{Tcharacterization-CCA} \Rightarrow \ref{Tcharacterization-strong}$) can fail when $G$ is not abelian. (So \cref{ChromIso} can also fail when $G$ is not abelian.) Counterexamples can be constructed on the dihedral group of order $4n$, whenever $n$ is odd \cite[Cor.~5.4 and Prop.~5.6]{HKMM} (and on infinitely many other finite groups).
    \item The implication ($\ref{Tcharacterization-CCA} \Rightarrow \ref{Tcharacterization-index}$) is a special case of \cite[Prop.~2.5]{HKMM}, which does not require $G$ to be abelian.
    \end{enumerate}
\end{rems}

Note that combining \cref{ChromIso,Tcharacterization} yields the following conclusion.

\begin{cor} \label{AllIsoAff}
If\/ $\Cay(G_1; S_1)$ and $\Cay(G_2; S_2)$ are connected abelian Cayley graphs, and $\Cay(G_1; S_1)$ is chromatically isomorphic to $\Cay(G_2; S_2)$, then the following are equivalent:
	\begin{enumerate}
	\item Every chromatic isomorphism from $\Cay(G_1; S_1)$ to $\Cay(G_2; S_2)$ is affine.
    	\item $|G_1: \langle T(S_1,2s)\rangle|\leq 2$ for all $s\in S_1$ with $|s|=4$. 
   	\item $|G_2: \langle T(S_2,2t)\rangle|\leq 2$ for all $t\in S_2$ with $|t|=4$. 
	\end{enumerate}
\end{cor}

The following result shows there is a unique 2-generated, non-CCA, connected abelian Cayley graph (up to affine isomorphism). This graph was first pointed out by G.\,Verret \cite[Example~2.2]{HKMM}.

\begin{cor}\label{2genAbelNotCCA}
    Assume $\{a,b\}$ is a 2-element generating set of an abelian group~$G$. Then the Cayley graph 
    $\Cay \bigl( G; \{\pm a, \pm b \} \bigr)$ 
    is not CCA if and only if
    $G\cong \Z_4\times \Z_2$ 
    and
    $|a| = |b| = 4$.
\end{cor}

The following result was known for abelian groups that are finite \cite[Prop.~4.1]{HKMM}, but the extension to infinite groups is new.

\begin{cor} \label{CCAiffFG}
A finitely generated abelian group~$G$ has a connected Cayley graph that is \underline{not} CCA \textup(or, equivalently, not strongly CCA\textup) if and only $G$ has a direct summand of the form\/ $\Z_4 \oplus \Z_2$ or\/ $\Z_{2^n} \oplus \Z_2 \oplus \Z_2$, with $n \ge 3$.
\end{cor}

See \cref{CCAIffOplus} for a slightly more complicated version of \cref{CCAiffFG} that does not require the abelian group to be finitely generated.

\begin{ack}
This research was partially supported by grants from the Natural Science and Engineering Research Council of Canada.
\end{ack}

\section{Preliminaries} \label{PrelimSect}

\begin{defn}
If $\varphi \colon \Cay(G_1; S_1) \to \Cay(G_2; S_2)$ is an isomorphism of abelian Cayley graphs, then, for each $g \in G_1$, the map $\varphi$ sends the neighbours of~$g$ bijectively onto the neighbours of~$\varphi(g)$. Hence, there is a bijection $\pi^\varphi_g \colon S_1 \to S_2$, such that
	\[ \varphi(g + s) = \varphi(g) + \pi^\varphi_g(s) \quad \text{for all $s \in S_1$.} \]
\end{defn}

The following \lcnamecref{pi} presents some simple properties of~$\pi^\varphi_g$. (This paper requires groups to be abelian, but we remark that the proofs of \cref{pi,morehalfaffine} do not require that assumption.)

\begin{lem} \label{pi}
Assume $\varphi \colon \Cay(G_1; S_1) \to \Cay(G_2; S_2)$ is an isomorphism of connected abelian Cayley graphs. Then:
	\begin{enumerate}
   
	\item \label{pi-affineiff}
	$\varphi$ is affine if and only if $\pi_g^\varphi$ is independent of~$g$: i.e., $\pi_g^\varphi(s) = \pi_h^\varphi(s)$ for all $g,h \in G_1$ and $s \in S_1$.
	   
	\item \label{pi-presiff}
	$\varphi$ is a colour-preserving automorphism if and only if $G_1 = G_2$, $S_1 = S_2$, and $\pi_g^\varphi(s) \in \{s, -s \}$ for all $g\in G_1$ and $s \in S_1$.

	\item \label{pi-chromiff}
	$\varphi$ is a chromatic isomorphism if and only if $\pi^\varphi_h(s) \in \{ \pm \pi^\varphi_g(s)\}$ 
	for all $g,h \in G_1$ and all $s \in S_1$. 
    
	\item \label{pi-chrom}
	If $\varphi$ is a chromatic isomorphism, then for all $g\in G_1$, $s\in S_1$, and $k\in \Z$, we have:
		\begin{enumerate}
		
		\item \label{pi-chrom-order}
		$|\pi^\varphi_g(s)| = |s|$,
                        
		\item \label{pi-chrom-coset}
		$\pi^\varphi_{g + ks}(s) = \pi_g^\varphi(s)$,
		and

		\item \label{pi-chrom-ks}
		$\varphi(g + ks) = \varphi(g) + k \,\pi^\varphi_g(s)$.   

 		\end{enumerate}
 	\end{enumerate}
\end{lem}

\begin{proof}
(\ref{pi-affineiff} $\Rightarrow$)
Write $\varphi(x) = \tau(x) + b$, where $\tau \colon G_1 \to G_2$ is a group isomorphism, and $b \in G_2$. Then
	\[ \varphi(g + s) 
	= \tau(g + s) + b 
	= \tau(g) + \tau(s) + b 
	= \bigl( \tau(g) + b \bigr)  + \tau(s)
	= \varphi(g) + \tau(s) , \]
so $\pi^\varphi_g(s) = \tau(s)$ is independent of~$g$.

\medbreak

(\ref{pi-affineiff} $\Leftarrow$)
By composing $\varphi$ with the translation $y \mapsto y - \varphi(0_{G_1})$, we may assume, without loss of generality, that $\varphi(0_{G_1}) = 0_{G_2}$. Then for all $x \in G_1$ and $s \in S_1$, we have
	\[ \pi^\varphi_x(s) = \pi^\varphi_0(s) = \varphi(s) . \]
Therefore, for all $s_1,\ldots,s_n \in S_1$, if we let $g = s_1 + s_2 + \cdots + s_{n-1}$, then
	\[ \varphi( s_1 + s_2 + \cdots + s_n )
		= \varphi(g) + \pi^\varphi_g(s_n)
		= \varphi(g) + \varphi(s_n) 
		= \varphi(s_1 + s_2 + \cdots + s_{n-1}) + \varphi(s_n) 
		. \]
By induction on~$n$, we conclude that 
	\[ \varphi(s_1 + s_2 + \cdots + s_n ) = \varphi(s_1) + \varphi(s_2) + \cdots + \varphi(s_n) . \]
Since $S_1$ generates~$G_1$, it follows easily from this that $\varphi$ is a group homomorphism, and hence $\varphi$ is affine.

\medbreak

\pref{pi-presiff}
Immediate from the definitions.

\medbreak

\pref{pi-chromiff}
Immediate from the definitions.

\medbreak

\pref{pi-chrom-order}
Since $\varphi$ is a chromatic isomorphism, it maps all of the edges of a single colour in $\Cay( G_1; S_1)$ to all of the edges of a certain colour in $\Cay( G_2; S_2)$. More precisely, $\varphi$~induces an isomorphism from $\Cay \bigl( G_1; \{\pm s\} \bigr)$ to $\Cay \bigl( G_2; \{\pm \pi^\varphi_g(s) \} \bigr)$, for each $s \in S_1$. Each component of the first graph has $|s|$ vertices, and each component of the second graph has $|\pi^\varphi_g(s)|$ vertices, so we must have $|s| = |\pi^\varphi_g(s)|$.

\medbreak

\pref{pi-chrom-coset}
It suffices to show $\pi^\varphi_{g + s}(s) = \pi^\varphi_g(s)$ for all~$g$ and~$s$, for then the map $g \mapsto \pi^\varphi_g$ must be constant on every orbit of the translation by~$s$, which means $\pi^\varphi_{g + ks}(s) = \pi^\varphi_g(s)$ for all $k \in \Z$.
Since $\varphi$ is a chromatic isomorphism, we have $\pi^\varphi_{g + s}(s) \in \{\pm \pi^\varphi_g(s) \}$. Therefore, if $\pi^\varphi_{g + s}(s) \neq \pi^\varphi_g(s)$, then we must have 
	\[ \pi^\varphi_{g + s}(s) = -\pi^\varphi_g(s) \neq \pi^\varphi_g(s) . \]
So
	\[ \varphi(g + 2s) 
	= \varphi(g + s) + \pi^\varphi_{g + s}(s)
	= \varphi(g + s) - \pi^\varphi_g(s)
	= \bigl( \varphi(g) + \pi^\varphi_g(s) \bigr) - \pi^\varphi_g(s)
	= \varphi(g)
	. \]
Since $\varphi$ is an isomorphism (hence injective), this implies $2s = 0$. So $2 \pi^\varphi_g(s) = 0$ by~\pref{pi-chrom-order}, which contradicts the fact that $-\pi^\varphi_g(s) \neq \pi^\varphi_g(s)$.

\medbreak

\pref{pi-chrom-ks}
The desired conclusion is obvious for $k = 0$.
We know from \pref{pi-chrom-coset} that 
	\[ \varphi(g + ks) = \varphi \bigl( g + (k - 1) s \bigr) + \pi^\varphi_g(s) , \]
so the desired conclusion can be proved by induction for $k > 0$.

Hence, for $k > 0$, we have
	\[ \varphi(g)
	= \varphi \bigl( (g - ks) + ks \bigr) 
	= \varphi( g - ks) + k \, \pi^\varphi_{g - ks}(s) 
	= \varphi(g - ks) + k \,  \pi^\varphi_{g}(s) 
	, \]
so $\varphi(g - ks) = \varphi(g) - k \, \pi^\varphi_{g}(s)$, which means that the desired conclusion is also true for negative~$k$.
\end{proof}

We will also use the following three 
observations.

\begin{lem}[{\cite[Lem.~2.7]{XiaZZ}}] \label{morehalfaffine}
Assume $\varphi \colon G \to G$ is an affine bijection. If the set of fixed points of~$\varphi$ is nonempty, then it is a coset of a subgroup of~$G$.
\end{lem}

\begin{proof}
We have $\varphi(x) = \tau(x) + b$, for some group automorphism $\tau$ of~$G$, and some $b \in G$. Let $F$ be the set of all elements of~$G$ that are fixed by~$\tau$. This is a subgroup of~$G$.

Now, if $g$ and~$h$ are fixed by~$\varphi$, then
	\[ g - h = \varphi(g) - \varphi(h) = \bigl( \tau(g) + b \bigr) - \bigl( \tau(h) + b \bigr) = \tau(g - h) , \]
so $g - h \in F$, which means $g \in h + F$. Conversely, a simple calculation shows that every element of the coset $h + F$ is fixed by~$\varphi$.
\end{proof}

\begin{lem} \label{complement}
Assume $H$ and $K$ are subgroups of~$G$. If $G \smallsetminus H$ is a coset of~$K$, then $|G : H| = 2$ \textup(and $H = K$\textup).
\end{lem}

\begin{proof}
Since $G \smallsetminus H$ is a coset of~$K$, we know that $H$ (the complement of $G \smallsetminus H$) is a union of cosets of~$K$, so $K \subseteq H$. Therefore, the coset $G \smallsetminus H$ of~$K$ is contained in a single coset of~$H$, so $|G : H| = 2$. (Furthermore, we must have $H = K$, because $G \smallsetminus H$ is a coset of both $H$ and~$K$.) 
\end{proof}

\begin{lem} \label{ChromIsoFG}
If \cref{ChromIso} is valid under the additional assumption that $S_1$ is finite, then it is also valid without this assumption.
\end{lem}

\begin{proof}
This is a straightforward compactness argument.
Given a chromatic isomorphism $\varphi$ from $\Cay(G_1; S_1)$ to $\Cay(G_2; S_2)$, with $\varphi(\0_{G_1}) = \0_{G_2}$, it suffices to find a colour-preserving automorphism~$\psi$ of $\Cay(G_1; S_1)$, with $\psi(\0_{G_1}) = \0_{G_1}$, such that the composition $\varphi \compose \psi$ is a group isomorphism.
We encode $\psi$ as the solution of a set of sentences of Propositional Logic: for each ordered pair $(g,h) \in G_1 \times G_1$, we have a Boolean variable $\var{g}{h}$, and we encode the fact that $\psi(g) = h$ by setting $\var{g}{h}$ true. Here are the sentences that must be satisfied:
	\begin{itemize}
	
	\item $\psi(\0) = \0$:
		\[ \var \0 \0 . \]
		
	\item $\psi$ is a well-defined injective function on its domain:
		\[ \bigl( \var g y \Rightarrow \neg \var g z \bigr)
			\and \bigl( \var y g \Rightarrow \neg \var z g \bigr)
			\qquad \text{for all $g,y,z \in G_1$ with $y \neq z$} . \]
	
	\item $\psi$ preserves edge colours and its domain is all of~$G_1$:
		\[ \var g h \Rightarrow \bigl( \var{g+s}{h+s} \vee \var{g+s}{h-s} \bigr)
			\qquad \text{for all $g,h \in G_1$ and $s \in S_1$} . \]
	
	\item $\psi$ is surjective (and its inverse preserves edge colours):
		\[ \var g h \Rightarrow \bigl( \var{g+s}{h+s} \vee \var{g-s}{h+s} \bigr)
			\qquad \text{for all $g,h \in G_1$ and $s \in S_1$} . \]
	
	\item $\varphi \compose \psi$ respects addition (so it is a group isomorphism):
		\[ \bigl( \var{g_1}{\varphi^{-1}(g_2)} \and \var{h_1}{\varphi^{-1}(h_2)})
			\Rightarrow \var{g_1 + h_1}{\varphi^{-1}(g_2 + h_2)}
			\qquad \text{for all $g_1, h_1 \in G_1$ and $g_2, h_2 \in G_2$} . \]
	\end{itemize}

If $F_1$ is any finite, symmetric subset of~$S_1$, and we let $F_2 = \varphi(F_1)$, then the restriction~$\varphi'$ of~$\varphi$ to $\langle F_1 \rangle$ is a chromatic isomorphism from $\Cay \bigl( \langle F_1 \rangle; F_1 \bigr)$ to $\Cay \bigl( \langle F_2 \rangle; F_2 \bigr)$ that sends $\0_{G_1}$ to~$\0_{G_2}$.
Our hypothesis implies that $\varphi'$ is the composition of 
	\begin{itemize}
	\item a colour-preserving automorphism of $\Cay \bigl( \langle F_1 \rangle; F_1 \bigr)$ that fixes~$\0_{G_1}$, 
	and 
	\item a group isomorphism from $\langle F_1 \rangle$ to $\langle F_2 \rangle$. 
	\end{itemize}
Therefore, every finite subset of the above collection of sentences is satisfiable, so the Compactness Theorem of Propositional Logic \cite[A.1.4.2, p.~26]{HandbookLogic} implies that the entire collection of sentences is satisfiable. This means that the desired colour-preserving automorphism~$\psi$ exists.
\end{proof}

\section{Proofs}

Recall that $T(S, 2s)$ is defined in the statement of \fullcref{Tcharacterization}{index}. 
The following result is our main technical tool.  It (together with \fullcref{pi}{chrom-coset}) implies that $\pi^\varphi_{g+h}(s) = \pi^\varphi_g(s)$ for all $h \in \langle T(S, 2s) \rangle$.

\begin{lem}\label{EltsOrder4}
Assume 
	\begin{itemize}
	\item $G_1$ and $G_2$ are abelian,
	\item $\varphi$ is a chromatic isomorphism from $\Cay(G_1;S_1)$ to $\Cay(G_2; S_2)$,
	and
	\item $\pi^\varphi_{g+t}(s) \neq \pi^\varphi_g(s)$ \textup(with $g \in G_1$ and $s,t \in S_1$\textup).
	\end{itemize} 
Then $|s| = |t| = 4$, $2s = 2t$, and $\langle s \rangle \cap \langle t \rangle = \langle 2s \rangle$.

Furthermore, we have $\pi^\varphi_{g+t}(s) =- \pi^\varphi_g(s)$ and $\pi^\varphi_{g+s}(t) = -\pi^\varphi_{g}(t) \neq \pi^\varphi_{g}(t)$.
\end{lem}

\begin{proof}
For convenience, let $\hat s = \pi_g^\varphi(s)$ and
$\hat t = \pi_g^\varphi(t)$, so
    \[ 
    \varphi(g + s) = \varphi(g) + \hat s
    \text{ and }
    \varphi(g + t) = \varphi(g) + \hat t.
    \]
Since $\pi^\varphi_{g+t}(s) \neq \pi^\varphi_g(s)$, but $\varphi$ is a chromatic isomorphism, we must have $\pi^\varphi_{g+t}(s) = -\hat s$. Therefore
	\begin{align*}
    \varphi(g) + \hat s + \pi^\varphi_{g+s}(t)
	&= \varphi(g + s) + \pi^\varphi_{g+s}(t)
	= \varphi(g + s  + t)
	\\&= \varphi(g + t  + s)
	= \varphi(g + t) + \pi^\varphi_{g+t}(s)
	= \varphi(g) + \hat t - \hat s
	, \end{align*}
so
	\[ \pi^\varphi_{g+s}(t) = \hat t - 2 \hat s . \]
Since $-\hat s = \pi^\varphi_{g+t}(s) \neq \pi^\varphi_g(s) = \hat s$, we know $2 \hat s \neq 0$, so we have $\pi^\varphi_{g+s}(t) \neq \hat t$, and hence $\pi^\varphi_{g+s}(t) = -\hat t$. So $-\hat t = \hat t - 2 \hat s$, which implies $2 \hat s = 2 \hat t$.

Also, we have 
	\[ \varphi(g + 2s) = \varphi(g) + 2 \hat s = \varphi(g) + 2 \hat t = \varphi(g + 2t) . \]
Since $\varphi$ is a bijection, this implies $2s = 2t$. 

Now
	\begin{align*}
	\varphi(g) + 3 \hat s
	&= \varphi(g + 3s)
	= \varphi(g + s + 2s)
	= \varphi(g + s + 2t)
	= \varphi(g + s)  + 2 \pi^\varphi_{g+s}(t)
	\\
	&= \varphi(g + s)  - 2 \hat t
	= \varphi(g + s)  - 2 \hat s
	= \bigl( \varphi(g) + \hat s \bigr) - 2 \hat s
	= \varphi(g) - \hat s
	. \end{align*}
So $4 \hat s = 0$. Since it was mentioned above that $2 \hat s \neq 0$, we conclude that $|\hat s| = 4$. 
As $2 \hat s = 2 \hat t$, this implies $|\hat t| = 4$. So $|s| = |t| = 4$ \fullcsee{pi}{chrom-order}.

Furthermore, since $\pi^\varphi_{g+t}(s) \neq \pi^\varphi_{g}(s)$, we know from \fullcref{pi}{chrom-coset} that $t \notin \langle s \rangle$. Since $|s| = |t| = 4$ and $2s = 2t$, it follows that $\langle s \rangle \cap \langle t \rangle = \langle 2s \rangle$.
\end{proof}

\begin{proof}[\bf Proof of \cref{ChromIso}]
Assume, for simplicity, that $S_1$ (and hence also $S_2$) is finite (so the groups $G_1$ and~$G_2$ are finitely generated).  \Cref{ChromIsoFG} explains that this special case implies the full result.

Let $\varphi$ be a chromatic isomorphism. By composing with the translation $x \mapsto x - \varphi(0_{G_1})$, we may assume, without loss of generality, that $\varphi(\0_{G_1}) = \0_{G_2}$. (Therefore $\pi^\varphi_0(s) = \varphi(s)$ for all $s \in S_1$.) For each $s \in S_1$, let 
	\[ A(\varphi, s) = \{\, g \in G_1 \mid \varphi(g + s) = \varphi(g) + \varphi(s) \,\} 
		= \{\, g \in G_1 \mid \pi^\varphi_g(s) = \pi^\varphi_0(s) \,\}. \]
Also let
	\[ \sgood{\varphi} = \{\, s \in S_1 \mid A(\varphi, s) = G_1 \,\} . \]
We wish to find a colour-preserving automorphism~$\alpha$ of $\Cay(G_1; S_1)$, such that $\alpha(0_{G_1}) = 0_{G_1}$ and $\sgood{\varphi \compose \alpha} = S_1$ (cf.\ \fullcref{pi}{affineiff}).

We may assume there is some $s \in S_1$, such that $s \notin \sgood{\varphi}$. Then (since $\Cay(G_1; S_1)$ is connected), there exist $h \in G_1$ and $t \in S_1$, such that 
	\[ \text{$\pi^\varphi_{h}(s) \neq \varphi(s)$ \ but \ $\pi^\varphi_{h+t}(s) = \varphi(s)$} . \]
We know from \cref{EltsOrder4} that
	\[ \text{$|s| = |t| = 4$, \ $2s = 2t$, \ and \ $\pi^\varphi_{h+s}(t) \neq \pi^\varphi_{h}(t)$} . \]
Let 
	\[ V = \{\, g \in G_1 \mid \text{$\pi^\varphi_g(s) \neq \varphi(s)$ and $\pi^\varphi_g(t) \neq \varphi(t)$} \,\} 
		= G_1 \smallsetminus \bigl( A(\varphi, s) \cup A(\varphi, t) \bigr) ,\]
and define $\alpha \colon G_1 \to G_1$ by 
	\[ \alpha(g) = \begin{cases}
		g + 2s & \text{if $g \in V$}; \\
		g & \text{otherwise}
		. \end{cases} \]

We break the remainder of the proof into short steps.

\refstepcounter{stepholder}
\begin{step}
The above function $\alpha$ is a colour-preserving automorphism of\/ $\Cay(G_1; S_1)$.
\end{step}
\fullCref{pi}{chrom-coset} implies that $A(\varphi, s) + 2s = A(\varphi, s)$. Since $2s = 2t$, we also have $A(\varphi, t) + 2s = A(\varphi, t)$. Therefore $V + 2s = V$, so it is easy to see that $\alpha$ is a bijection.

All that remains is to show that $\alpha$ preserves the colour of every edge. To this end, let $g_1$ and~$g_2$ be two adjacent vertices of $\Cay(G_1; S_1)$. If both of the vertices are in~$V$ or neither of the two vertices are in~$V$, then $\alpha(g_2) - \alpha(g_1) = g_2 - g_1$, so $\alpha$ preserves the colour of the edge from $g_1$ to~$g_2$.  We can now assume that $g_1 \notin V$ and $g_2 \in V$. Since $g_1 \notin V$, we may assume (perhaps after interchanging $s$ and~$t$) that $\pi^\varphi_{g_1}(s) = \varphi(s)$. However, since $g_2 \in V$, we know $\pi^\varphi_{g_2}(s) \neq \varphi(s)$. Therefore, we have $\pi^\varphi_{g_1}(s) \neq \pi^\varphi_{g_2}(s)$, so \cref{EltsOrder4} implies $|g_2 - g_1| = 4$ and $2(g_2 - g_1) = 2s$. Then
	\[ \alpha(g_2) - \alpha(g_1) 
		= \bigl( g_2 + 2s \bigr) - g_1 
		= \bigl( g_2 + 2(g_2 - g_1) \bigr) - g_1 
		= 3(g_2 - g_1)
		= - (g_2 - g_1)
		, \]
so $\alpha$ preserves the colour of the edge from $g_1$ to~$g_2$ in this case, as well.

\begin{step}
We have $\sgood{\varphi} \subseteq \sgood{\varphi \compose \alpha}$.
\end{step}
Let $r \in \sgood{\varphi}$, and let $g \in G_1$. Then $\pi^\varphi_g(r) = \pi^\varphi_{g+s}(r) = \pi^\varphi_{g+t}(r)$, so we see from \cref{EltsOrder4} that $\pi^\varphi_{g+r}(s) = \pi^\varphi_{g}(s)$ and $\pi^\varphi_{g+r}(t) = \pi^\varphi_{g}(t)$.  Therefore, either $g$ and $g + r$ are both in~$V$, or neither is.  Hence, there exists $w \in \{0_{G_1}, 2s\}$, such that $\alpha(g) = g + w$ and $\alpha(g + r) = g + r + w$. 

	We claim that, for all $x \in G_1$ (including $g$ and $g + r$), we have $\varphi(x + w) = \varphi(x) + \varphi(w)$. This is obvious if $w = 0_{G_1}$, so assume $w = 2s$. Then, using the fact that $|s| = 4$, so $-2s = 2s$, we have
	\[ \varphi(x + w) = \varphi(x + 2s) = \varphi(x) \pm \varphi(2s) = \varphi(x) + \varphi(2s) = \varphi(x) + \varphi(w) . \]
This completes the proof of the claim.

Now
	\begin{align*}
	(\varphi \compose \alpha)(g + r) - (\varphi \compose \alpha)(g) 
	&= \varphi \bigl( \alpha(g + r) \bigr) - \varphi \bigl( \alpha(g) \bigr)
	\\&= \varphi( g + r + w) - \varphi ( g + w )
	\\&= \bigl( \varphi( g + r ) + \varphi(w) \bigr) - \bigl( \varphi ( g )  + \varphi(w) \bigr)
	\\&= \bigl( \varphi(g) + \varphi(r) + \varphi(w) \bigr) - \bigl( \varphi ( g )  + \varphi(w) \bigr)
	\\&= \varphi(r)
	, \end{align*}
so $r \in \sgood{\varphi \compose \alpha}$.

\begin{step}
We have $A(\varphi, s) \subsetneq A(\varphi \compose \alpha, s)$.
\end{step}
Let $g \in A(\varphi, s)$, so $\pi^\varphi_g(s) = \varphi(s)$. Then, by \fullcref{pi}{chrom-coset}, we also have $\pi^\varphi_{g+s}(s) = \varphi(s)$. So neither $g$ nor $g + s$ belongs to~$V$. Therefore $\alpha$ fixes $g$ and $g + s$, so 
	\[ (\varphi \compose \alpha)(g + s) - (\varphi \compose \alpha)(g)
		= \varphi(g + s) - \varphi(g)
		= \varphi(s) , \]
so $g \in A(\varphi \compose \alpha, s)$. This establishes that $A(\varphi, s) \subseteq A(\varphi \compose \alpha, s)$.

We now show that the inclusion is proper.
Recall that $\pi^\varphi_{h+s}(t) \neq \pi^\varphi_{h}(t)$. So we can assume, without loss of generality, that $\pi^\varphi_{h+s}(t) = -\varphi(t)$ and $\pi^\varphi_{h}(t) = \varphi(t)$. (If necessary, replace $h$ with $h + s$ and replace $s$ with~$-s$.) By the choice of~$h$, we know $\pi^\varphi_{h}(s) = -\varphi(s)$, so \fullcref{pi}{chrom-coset} implies $\pi^\varphi_{h+s}(s) = -\varphi(s)$. Therefore $h + s \in V$. On the other hand, $h \notin V$, because $\pi^\varphi_{h}(t) = \varphi(t)$. Therefore
	\[ \varphi \bigl( \alpha(h + s) \bigr) - \varphi \bigl( \alpha(h) \bigr) 
		=  \varphi \bigl( h + s + 2s \bigr) - \varphi \bigl( h \bigr) 
		=  3 \pi^\varphi_h(s)
		=  - \pi^\varphi_h(s)
		= \varphi(s) , \]
so $h \in A(\varphi \compose \alpha, s)$. However, $h \notin A(\varphi, s)$ (because $\pi^\varphi_h(s) = - \varphi(s)$). This establishes that the inclusion $A(\varphi, s) \subseteq A(\varphi \compose \alpha, s)$ is proper. 

\begin{step}
Completion of the proof.
\end{step}
Suppose, for a contradiction, that it is impossible to find a colour-preserving automorphism~$\alpha$ of $\Cay(G_1; S_1)$ (with $\alpha(0_{G_1}) = 0_{G_1}$), such that $S_1 \subseteq \sgood{\varphi \compose \alpha}$. Then, letting $S_1 = \{s_1,\ldots,s_n\}$, there exist $k \le n$, and a colour-preserving automorphism~$\alpha$ of $\Cay(G_1; S_1)$, such that $\{s_1,\ldots,s_{k-1}\} \subseteq \sgood{\varphi \compose \alpha}$ (and $\alpha(0_{G_1}) = 0_{G_1}$), but there does not exist such an automorphism for $\{s_1,\ldots,s_{k}\}$. To simplify notation, assume $\{s_1,\ldots,s_{k-1}\} \subseteq \sgood{\varphi }$ (by replacing $\varphi$ with $\varphi \compose \alpha$).

Now, we inductively construct a sequence $\alpha_0, \alpha_1, \ldots$ of colour-preserving automorphisms of the Cayley graph $\Cay(G_1; S_1)$, such that for each~$i$:
	\begin{itemize}
	\item $\alpha_i(0_{G_1}) = 0_{G_1}$,
	\item $\{s_1,\ldots,s_{k-1}\} \subseteq \sgood{\varphi \compose \alpha_i}$,
	and
	\item $A(\varphi \compose \alpha_i, s_k) \subsetneq A(\varphi \compose \alpha_{i+1}, s_k)$.
	\end{itemize}
Let $\alpha_0$ be the identity map. Given $\alpha_i$, we know $s_k \notin \sgood{\varphi \compose \alpha_i}$, so
	Steps 1--3 
provide a colour-preserving automorphism~$\alpha$, such that $\sgood{\varphi \compose \alpha_i} \subseteq \sgood{\varphi \compose \alpha_i \compose \alpha}$ and 
 \[ A(\varphi \compose \alpha_i, s) \subsetneq A(\varphi \compose \alpha_i \compose \alpha, s_k) . \] 
 Let $\alpha_{i+1} = \alpha_i \compose \alpha$.

However, it follows from \cref{EltsOrder4} that, for each~$i$, the set $A(\varphi \compose \alpha_i, s_k)$ is a union of cosets of $H \coloneqq \langle T(S_1, 2 s_k) \rangle$. Also, from the definition of $T(S_1, 2 s_k)$, it is clear that $2S_1 \subseteq H$, so $2G_1 \subseteq H$. Since $G_1$ is finitely generated (recall that we are assuming $S_1$ is finite), this implies $|G_1 : H|$ is finite. So there are only finitely many possibilities for $A(\varphi \compose \alpha_i, s_k)$, which contradicts the existence of the infinite sequence 
	\[ A(\varphi \compose \alpha_0, s_k) \subsetneq A(\varphi \compose \alpha_1, s_k) \subsetneq A(\varphi \compose \alpha_2, s_k) \subsetneq \cdots
	. \qedhere \]
\end{proof}

\goodbreak 

\begin{rems}
\leavevmode
	\begin{enumerate}
	\item \Cref{ChromIso} tells us that every chromatic isomorphism~$\varphi$ from $\Cay(G_1; S_1)$ to $\Cay(G_2; S_2)$ is of the form $\tau \compose \alpha$, where $\tau$ is a group isomorphism from $G_1$ to~$G_2$, and $\alpha$ is a colour-preserving automorphism of $\Cay(G_1; S_1)$.
By applying \cref{ChromIso} to the inverse of~$\varphi$, we see that $\varphi$ can also be written in the form $\alpha' \compose \tau'$, where $\tau'$ is a group isomorphism  from $G_1$ to~$G_2$, and $\alpha'$ is a colour-preserving automorphism of $\Cay(G_2; S_2)$.

	\item The proof of \cref{ChromIso} provides additional information about colour-preserving automorphisms.  In particular, every colour-preserving automorphism is the composition of an affine automorphism and a colour-preserving automorphism $\alpha$, such that 
		\begin{itemize}
		\item $\alpha(0_{G_1}) = 0_{G_1}$,
		and
		\item $|\alpha(g) - g| \le 2$ for all $g \in G_1$.
		\end{itemize}
	This implies that $\alpha$ fixes every element of~$S_1$ that does not have order~$4$.
	\end{enumerate}
\end{rems}

\begin{proof}[\bf Proof of \cref{Tcharacterization}]
($\ref{Tcharacterization-CCA} \Rightarrow \ref{Tcharacterization-strong}$)
Let $\varphi$ be a colour-permuting automorphism. By \cref{ChromIso}, it is the composition of a colour-preserving automorphism and a group isomorphism. However, since the Cayley graph is CCA, every colour-preserving automorphism is affine. Therefore, $\varphi$ is the composition of two affine maps, and is therefore affine.

\medbreak

($\ref{Tcharacterization-strong} \Rightarrow \ref{Tcharacterization-CCA}$)
Obvious.

\medbreak

($\ref{Tcharacterization-index} \Rightarrow \ref{Tcharacterization-CCA}$) 
We prove the contrapositive: assume there is a colour-preserving automorphism~$\varphi$ of $\Cay(G;S)$ that is not affine. Since $\Cay(G;S)$ is connected, this implies there exist $g \in G$ and $s,t \in S$, such that $\pi^\varphi_{g + t}(s) \neq \pi^\varphi_g(s)$. For convenience, let $H = \langle T(S,2s)\rangle$. It follows from Lemmas 
\ref{EltsOrder4} and \fullref{pi}{chrom-coset} that $\pi^\varphi_{g + h}(s) = \pi^\varphi_g(s)$ for all $h \in \langle H, s \rangle$, so we must have $t \notin \langle H, s \rangle$. Similarly, since \cref{EltsOrder4} tells us that $\pi^\varphi_{g + s}(t) \neq \pi^\varphi_g(t)$, and that $2t = 2s$ (so $\langle T(S, 2t) \rangle = \langle T(S,2s) \rangle = H$), we must have $s \notin \langle t, H \rangle$. Then $H$, $s + H$, and $t + H$ are three different cosets of~$H$, so $|G : H| \ge 3$.

\medbreak

($\ref{Tcharacterization-CCA} \Rightarrow \ref{Tcharacterization-index}$)
We again prove the contrapositive: assume 
	    \[ \text{$|G: \langle T(S,2s)\rangle| > 2$ for some $s\in S$ with $|s| = 4$.} \]
Define $\varphi \colon G \to G$ by
\[
\varphi(x) =
    \begin{cases}
        x + 2s & \text{if $x \in \langle T(S, 2s) \rangle$} , \\
        x & \text{otherwise},
    \end{cases}
\]
Since $2s \in \langle T(S, 2s) \rangle$, it is easy to see that $\varphi$ is a bijection.

Also, since $\lvert G : \langle T(S,c)\rangle\rvert > 2$, it is easy to see that the set of fixed points of~$\varphi$ (i.e., the complement of $\langle T(S,c)\rangle$) is not a coset of any subgroup of~$G$ \csee{complement}. Therefore, we see from \cref{morehalfaffine} that $\varphi$ is not affine.  

Hence, it suffices to show $\varphi$ is colour-preserving. (For then $\Cay(G;S)$ has a colour-preserving automorphism that is not affine, so $\Cay(G; S)$ is not CCA.) Given $g,h \in G$ with $g - h \in S$, we wish to show $\varphi(g) - \varphi(h) \in \{\pm (g - h)\}$. By the definition of~$\varphi$, there exist $c_g,c_h \in \{0,2s\}$, such that $\varphi(g) = g + c_g$ and $\varphi(h) = h + c_h$.  If $c_g = c_h$, then 
	\[ \varphi(g) - \varphi(h) = (g + c_g) - (h + c_h) = g - h \in \{\pm (g - h)\} , \]
as desired. 

So we may assume $c_g \neq c_h$. 
Assume without loss of generality that $g \in \langle T(S, 2s) \rangle$ and $h \notin \langle T(S, 2s) \rangle$. Then $t \coloneqq g - h \notin \langle T(S, 2s) \rangle$. Since $t \in S$, then the definition of $T(S, 2s)$ implies $2t = 2s$ (and hence $|t| = 4$). Then
	\[ \varphi(g) - \varphi(h)
	= (g + 2s) - h
	= (g - h) + 2s
	= t + 2t
	= 3t
	= -t 
	= -(g-h)
	\in  \{\pm (g - h)\}
	, \]
as desired.
\end{proof}

\begin{proof}[\bf Proof of \cref{2genAbelNotCCA}]
($\Rightarrow$)
Since the Cayley graph $\Cay \bigl( G; \{\pm a, \pm b\} \bigr)$ is not CCA, there is colour-preserving automorphism $\varphi$ that is not affine.
Since $\{a,b\}$ generates~$G$, then there exist $s,t\in\{a,b\}$ and $g \in G$ with $\pi_g^\varphi(s)\neq\pi_{g+t}^\varphi(s)$. 
So, \Cref{EltsOrder4} implies that $2a=2b$,  $|a|=|b|=4$, and  $a \notin \langle b \rangle$. The only abelian group with two generators satisfying these conditions is $\Z_4\times\Z_2$. 

($\Leftarrow$)
Let $G \cong \Z_4\times \Z_2$. 
The elements of order $4$ in $G$ are 
$\{ \pm(1,0), \pm(1,1) \}$, so there is only one generating set $\{ \pm a, \pm b \}$ with $|a| = |b| = 4$ that needs to be considered: we show that the abelian Cayley graph $\Cay(\Z_4 \times \Z_2; S)$ is not CCA for $S = \{\pm(1,0), \pm(1,1)\}$.
This can easily be accomplished by exhibiting a specific colour-preserving automorphism that is not affine \cite[Example~2.2]{HKMM}, but we will use this as an opportunity to demonstrate \cref{Tcharacterization}.

For every $s \in S$ we have $2s = (2,0)$. Therefore $T(S,2s) = \{2s\} = \{(2,0)\}$, and hence $\langle T(S,2s)\rangle = \{(0,0),(2,0)\}$ has order $2$. Therefore $|G:\langle T(S,2s) \rangle| = 4 > 2$, so we see from \cref{Tcharacterization} that the Cayley graph is not CCA.
\end{proof}

\begin{cor} \label{CCAIffOplus}
Assume $G$ is an abelian group. Then the following are equivalent:
\noprelistbreak
	\begin{enumerate}
	\item \label{CCAIffOplus-CCA}
	$G$ has a connected Cayley graph that is \underline{not} CCA.
	\item \label{CCAIffOplus-strong}
	$G$ has a connected Cayley graph that is \underline{not} strongly CCA.
	\item \label{CCAIffOplus-4}
	$G$ has a subgroup~$H$, such that either
		\begin{itemize}
		\item $G \iso H \oplus \Z_4 \oplus \Z_2$, 
		or
		\item $G \iso H \oplus \Z_2 \oplus \Z_2$, and $H$ has an element of order~$4$.
		\end{itemize}
	\item \label{CCAIffOplus-8}
	$G$ has a subgroup~$H$, such that either
		\begin{itemize}
		\item $G \iso H \oplus \Z_4 \oplus \Z_2$, 
		or
		\item $G \iso H \oplus \Z_2 \oplus \Z_2$, and $H$ has an element of order~$8$.
		\end{itemize}
	\end{enumerate}
\end{cor}

\begin{proof}
($\ref{CCAIffOplus-CCA} \Leftrightarrow \ref{CCAIffOplus-strong}$)
Immediate from ($\ref{Tcharacterization-CCA} \Leftrightarrow \ref{Tcharacterization-strong}$) of \cref{Tcharacterization}.

\medbreak

($\ref{CCAIffOplus-strong} \Rightarrow \ref{CCAIffOplus-4}$)
By \cref{Tcharacterization}, we know $|G: \langle T(S,2s)\rangle| > 2$ for some $s\in S$ with $|s|=4$. Since $S$ generates~$G$, there is some $t \in S$, such that $t \notin \langle s, T(S,2s) \rangle$ (so $2t = 2s$). 
The group $\overline{G} \coloneqq G/\langle T(S,2s)\rangle$ is an elementary abelian $2$-group (i.e., a vector space over~$\Z_2$), so we can write 
	\[ \overline{G} = \overline{H} \oplus \langle \overline{s}, \overline{t} \rangle , \]
for some subgroup $\overline{H}$ of~$\overline{G}$. Let $H$ be the pullback of~$\overline{H}$ to a subgroup of~$G$ that contains $\langle T(S,2s)\rangle$, and note that $H \cap \langle s, t \rangle = \langle 2s \rangle$. 

If there is some $h \in H$, such that $2h = 2s$, then we have
	\[ G = H \oplus \langle s - h, t - h \rangle \iso H \oplus \Z_2 \oplus \Z_2 , \]
and $H$ contains the element~$h$ of order~$4$.

So we can assume $2s \notin 2H$. Then, since $H/(2H)$ is a vector space, it is not difficult to show there is a subgroup~$H_0$ of~$H$, such that $H = H_0 \oplus \langle 2s \rangle$. So
	\[ G = H_0 \oplus \langle s,t \rangle \iso H_0 \oplus \Z_4 \oplus \Z_2 , \]
as desired. 

\medbreak

($\ref{CCAIffOplus-4} \Rightarrow \ref{CCAIffOplus-strong}$)
To simplify notation, assume $G$ is {equal} to the direct sum, rather than just being isomorphic to it.
Let $S = H \cup \{s,t\}$, where
	\begin{itemize}
	\item $s = (0,1,0)$ and $t = (0,1,1)$ if $G = H \oplus \Z_4 \oplus \Z_2$, 
	and
	\item $s = (a, 1, 0)$ and $t = (a, 0,1)$ if $G = H \oplus \Z_2 \oplus \Z_2$, and $H$ has an element~$a$ of order~$4$.
	\end{itemize}
Then $|s| = 4$ and $2s = 2t$, so 
	\[ \langle T(S, 2s) \rangle \subseteq \langle H, 2s \rangle . \]
In both cases, it is easy to verify that $G /  \langle H, 2s \rangle \iso \Z_2 \times \Z_2$, so we conclude from \cref{Tcharacterization} that $\Cay(G;S)$ is not strongly CCA.

\medbreak

($\ref{CCAIffOplus-4} \Rightarrow \ref{CCAIffOplus-8}$)
We may assume $G = H \oplus \Z_2 \oplus \Z_2$, and that $H$ has an element~$a$ of order~$4$, but does not have an element of order~$8$. Since $H$ has no element of order~$8$, it is not difficult to see that $\langle a \rangle$ is a direct summand of~$H$: we have $H = H_0 \oplus \langle a \rangle$, for some subgroup~$H_0$. (For example, choose a complement $\overline{H_1}$ to $\overline{a}$ in the vector space $H/(2H)$. Then choose a complement $\overline{H_0}$ to $\overline{2a}$ in the vector space $(2H)/(4H)$.) Then $G = H_0 \oplus \langle a \rangle \oplus \Z_2 \oplus \Z_2$. Since $\langle a \rangle \iso \Z_4$, this group obviously has a direct summand that is isomorphic to $\Z_4 \oplus \Z_2$. So $G$ appears in the first case of~\pref{CCAIffOplus-8}.

\medbreak

($\ref{CCAIffOplus-8} \Rightarrow \ref{CCAIffOplus-4}$)
This is trivial: every group with an element of order~$8$ also has an element of order~$4$.
\end{proof}

\begin{proof}[\bf Proof of \cref{CCAiffFG}]
($\Leftarrow$) Immediate from ($\ref{CCAIffOplus-8} \Rightarrow \ref{CCAIffOplus-CCA}$) of \cref{CCAIffOplus}.

\medbreak

($\Rightarrow$) 
By \fullcref{CCAIffOplus}{8}, we may assume $G$ has a subgroup~$H$, such that $G \iso H \oplus \Z_2 \oplus \Z_2$, and $H$ has an element of order~$8$. Since $G$ is a finitely generated abelian group, the subgroup~$H$ is also finitely generated \cite[p.~82]{Fuchs}. So $H$ is a direct sum of finitely many cyclic groups, such that each of these cyclic subgroups either has prime-power order, or is infinite \cite[Thm.~3.2.7, p.~84]{Fuchs}: 
	\[ H = C_1 \oplus C_2 \oplus \cdots \oplus C_r . \]
Since $H$ has an element of order~$8$, we may assume (by permuting the factors) that $C_r$ is finite, and $|C_r|$ is divisible by~$8$. Then, since $|C_r|$ is a prime-power, we have $|C_r| = 2^n$, for some~$n$. Furthermore, $n \ge 3$ since $|C_r|$ is divisible by~$8$. Then
	\[ G \iso H \oplus \Z_2 \oplus \Z_2 \iso C_1 \oplus C_2 \oplus \cdots \oplus C_{r-1} \oplus \Z_{2^n} \oplus \Z_2 \oplus \Z_2 , \]
so $G$ has $\Z_{2^n} \oplus \Z_2 \oplus \Z_2$ as a direct summand.
\end{proof}

\end{document}